\documentclass[11pt]{amsart}

\usepackage[T1]{fontenc}
\usepackage{lmodern}
\usepackage{microtype}
\usepackage{amsmath,amssymb,amsfonts,mathtools}
\usepackage{enumitem}
\usepackage{xcolor}
\usepackage[colorlinks=true,linkcolor=blue!45!black,citecolor=blue!45!black,urlcolor=blue!55!black]{hyperref}
\usepackage[nameinlink,capitalise,noabbrev]{cleveref}

\hypersetup{
  pdftitle={Mathematical Koans and Cartan Convexity: Gamma-Convex Hulls and Butterfly Realizations},
  pdfauthor={J. E. Pascoe},
  pdfsubject={Mathematical koans, local-to-global extension, Gamma-convexity, plurisubharmonicity, and butterfly realizations},
  pdfkeywords={mathematical koan, free function, matrix convexity, Gamma-convexity, plurisubharmonicity, convex hull, butterfly realization}
}

\allowdisplaybreaks
\setlist[itemize]{leftmargin=2em}
\setlist[enumerate]{leftmargin=2.2em}

\newtheorem{theorem}{Theorem}[section]
\newtheorem{proposition}[theorem]{Proposition}
\newtheorem{lemma}[theorem]{Lemma}
\newtheorem{corollary}[theorem]{Corollary}
\newtheorem{question}[theorem]{Question}
\theoremstyle{definition}
\newtheorem{definition}[theorem]{Definition}
\newtheorem{example}[theorem]{Example}
\theoremstyle{remark}
\newtheorem{remark}[theorem]{Remark}

\newcommand{\Hh}{\mathcal H}
\newcommand{\Kk}{\mathcal K}
\newcommand{\Ss}{\mathbb S}
\newcommand{\ncco}{\operatorname{co}_{\mathrm{nc}}}
\newcommand{\sot}{\mathrm{SOT}}
\newcommand{\ran}{\operatorname{ran}}
\newcommand{\Gco}{\Gamma\text{-}\operatorname{co}_{\mathrm{nc}}}

\title[Mathematical koans and Cartan convexity]{Mathematical Koans and Cartan Convexity:\\
$\Gamma$-Convex Hulls and Butterfly Realizations}
\author{J. E. Pascoe}
\date{August 2026}


\begin{document}

\begin{abstract}
In homage to Hofstadter, we call a compact, citable specification of
definitions, theorem, proof mechanism, and examples a \emph{mathematical
koan} when a full paper can be reconstructed and verified from it, including
in reader-specific forms generated by AI. Every AI-solvable problem is itself
a koan, although reconstruction need not be cheap; the present article is an
expansion of one.

We introduce Cartan convexity for self-adjoint free functions: locally, such a
function agrees with a locally bounded matrix-convex free function. If the
variable set has cardinality $\tau$ and
$\lambda=\max\{\tau,\aleph_0\}$, every bounded real free set has a universal
direct sum on a Hilbert space of dimension at most $2^\lambda$; every point of
the set is a reducing summand of an amplification of this sum. Applying the
noncommutative Kraus--butterfly theorem at that universal point yields an
extension to an open matrix-convex neighborhood of the noncommutative convex
hull. For normal affine pencils, the extension domain also contains the
bounded strong closure of the hull.

We prove an analogous theorem for a graph embedding $\Gamma$. A local convex
lift through $\Gamma$ extends to a neighborhood of
$\Gamma^{-1}(\operatorname{co}_{\mathrm{nc}}\Gamma(K))$ and admits a
butterfly realization in the $\Gamma$-coordinates. We distinguish this lift
condition from intrinsic $\Gamma$-convexity. For
$\Gamma(x,y)=(x,y,y^2)$, the intrinsically $\Gamma$-affine polynomial
$xy+yx$ has no convex lift germ at the origin, whereas a single quadratic
coordinate suffices to lift every uniformly real analytic germ.

For the graph consisting of all analytic and coanalytic monomials, intrinsic
$\Gamma$-convexity is equivalent to free plurisubharmonicity in the polynomial
and uniformly analytic settings, and the scalar hull is the classical
polynomial hull. For $\Gamma_J(x,y)=(x,y,xy+yx)$, admissible compressions are
characterized by vanishing Jordan covariance. At level one the associated
separating curves are rectangular hyperbolas, and the sharp scalar
Carath\'eodory number is four.
\end{abstract}

\maketitle

\section{Introduction}

The central observation of this paper is that a real free set allows a
family of local hypotheses to be represented by a single local hypothesis,
but the universal point must be constructed at a definite Hilbert-space
level. Let $\mathcal I$ be the set of variables and put
\[
 \lambda=\max\{|\mathcal I|,\aleph_0\},\qquad \mu=2^\lambda.
\]
Every tuple decomposes into reducing summands, each acting on a Hilbert space
of dimension at most $\lambda$. On the standard spaces $\ell_2(\kappa)$,
$\kappa\leq\lambda$, there are at most $\mu$ such tuples altogether. For a
bounded real free set $K$, sum all of its points at these standard levels:
\[
 K_{\leq\lambda}=\bigcup_{\kappa\leq\lambda}K(\ell_2(\kappa)),
 \qquad
 \mathbf X=\bigoplus_{X\in K_{\leq\lambda}}X.
\]
Then $\mathbf X\in K$ acts on a Hilbert space of dimension at most $\mu$, and
every point of $K$ is a reducing summand of an amplification of $\mathbf X$.
Consequently, a convex lift near $\mathbf X$ simultaneously determines
compatible data at every point of $K$.

Matrix convexity is particularly suited to this construction. Its domains are
closed under direct sums and isometric compressions, its locally bounded
functions are real analytic, and its realization theory provides formulas of
the form
\begin{equation}\label{eq:intro-butterfly}
 f(X)=\ell(X)+\Lambda(X)^*P(X)^{-1}\Lambda(X),
\end{equation}
where $\ell$, $\Lambda$, and the self-adjoint pencil $P$ are affine and
$P(X)\succ0$. The epigraph of this function is described by the linear matrix
inequality
\begin{equation}\label{eq:intro-epi}
 \begin{pmatrix}
  Y-\ell(X)&\Lambda(X)^*\\
  \Lambda(X)&P(X)
 \end{pmatrix}\succeq0.
\end{equation}

The affine structure also propagates a uniform positive lower bound through
all noncommutative convex combinations. Since every point of $\ncco(K)$ is an
isometric compression of an amplification of $\mathbf X$, the realization at
$\mathbf X$ naturally extends to a neighborhood of the entire hull.

We call a self-adjoint free function $f$ on $K$ \emph{Cartan convex} when,
near every point of $K$, it agrees with a locally bounded matrix-convex free
function. The terminology refers to the local-to-envelope viewpoint of
Cartan--Thullen theory \cite{CartanThullen}; no equivalence with the classical
theory of domains of holomorphy is asserted. Precise definitions appear in
\cref{sec:all-level}.

\begin{theorem}[Cartan extension theorem; informal form]\label{thm:intro-cartan}
Let $K$ be a bounded real free set in a set $\mathcal I$ of variables.
Every Cartan convex free function $f$ on $K$ extends to a locally bounded
matrix-convex free function on an open matrix-convex neighborhood of the
noncommutative convex hull of $K$. The extension admits a
butterfly realization of the form \eqref{eq:intro-butterfly}. If its affine
maps are bounded-SOT continuous, the neighborhood also contains the bounded
strong closure of the hull.
\end{theorem}

The proof uses the universal point $\mathbf X$ and the convex lift at that
point. The noncommutative Kraus--butterfly theorem of Pascoe and Tully-Doyle
\cite{PTD} gives a realization with
$P(\mathbf X)\succeq\varepsilon I$. Affineness propagates this inequality to
every compression of every amplification of $\mathbf X$, hence to the
noncommutative convex hull.
Formula \eqref{eq:intro-butterfly} therefore continues the lift to
$\{P\succ0\}$. Every small component of every $X\in K$ occurs directly as a
reducing summand of $\mathbf X$; equality on those components recovers $f(X)$.

The second theorem passes the same mechanism through a graph map
$\Gamma=(\gamma_1,\dots,\gamma_r)$ of self-adjoint free functions containing
the original coordinates. A nonlinear $\Gamma$ specifies the information
preserved by an admissible compression. The foundational theory of
$\Gamma$-convexity \cite{JKMMPGamma} shows that the resulting hull is
the inverse image of an ordinary matrix-convex hull upstairs. We call $f$
\emph{lift-Cartan $\Gamma$-convex} when the induced function on $\Gamma(K)$
is Cartan convex; equivalently, locally
\[
 f(X)=g_X(\Gamma(X))
\]
for an ordinary matrix-convex free function $g_X$ in the lifted coordinates.

\begin{theorem}[$\Gamma$-Cartan extension theorem; informal form]
\label{thm:intro-gamma}
Let $K$ and $\Gamma$ be as above. Every lift-Cartan $\Gamma$-convex function
extends to an open $\Gamma$-convex neighborhood of $\Gco(K)$ and has a
$\Gamma$-butterfly realization
\begin{equation}\label{eq:intro-gamma-butterfly}
 f(X)=\ell(\Gamma(X))+
 \Lambda(\Gamma(X))^*P(\Gamma(X))^{-1}\Lambda(\Gamma(X)).
\end{equation}
Its epigraph is cut out by the corresponding block $\Gamma$-pencil.
\end{theorem}

The lift-Cartan condition is stronger than the intrinsic Jensen inequality,
which tests only compressions $V$ satisfying
\[
 V^*\Gamma(X)V=\Gamma(V^*XV).
\]
Already for partial convexity,
$\Gamma_{\mathrm{par}}(x,y)=(x,y,y^2)$, the polynomial $xy+yx$ satisfies the
intrinsic inequality with equality but has no local matrix-convex lift through
the $\Gamma_{\mathrm{par}}$-graph. The obstruction is a nonzero mixed Hessian
along the kernel of the $x$-Hessian; it is also visible in the affine tail of
the root-butterfly realization from \cite{JKMMPPartial}.
Thus preservation of the intrinsic Jensen inequalities does not by itself
produce a convex extension away from the graph.

The free-set hypothesis is essential. It replaces compactness and patching
by a universal direct sum. At finitely many levels such a universal point need
not exist, and a three-point scalar example shows that the theorem can fail. This
operator-level emphasis is
consistent with
\cite{JKMMPGamma,KMS} and with the all-cardinal formulation of
noncommutative convex sets in \cite{DK,HKM}.

The examples illustrate several different forms of $\Gamma$-geometry.
If
\[
 \Gamma_{\square}(X)=
 \left(X_1,\dots,X_g,\sum_{j=1}^gX_j^2\right),
\]
then every $\Gamma_{\square}$-pair is a reducing compression. Intrinsic
$\Gamma_{\square}$-convexity is therefore vacuous: every free function passes
the Jensen test because there are no nonreducing tests left to fail. The
lifting problem can also be solved. A middle-matrix bound for
the nc Hessian supplies a completely positive correction in the normal
coordinate $\sum X_j^2-S$, and a second square controls the mixed normal
directions. Thus one quadratic coordinate suffices to convexify every uniformly
real analytic germ, even though the intrinsic condition is vacuous.

The full holomorphic graph behaves in the opposite way. Write $Z=X+iY$ and
let
$\Gamma_{\mathrm{hol}}(Z)$ consist of the real and imaginary parts of every
analytic word $Z^w$, equivalently of all analytic and coanalytic monomials.
Upper and lower triangular perturbations of $Z\oplus Z$ become
$\Gamma_{\mathrm{hol}}$-pairs, and Jensen's
inequality on their first corners splits the nc complex Hessian into its two
positive orderings. Intrinsic $\Gamma_{\mathrm{hol}}$-convexity therefore
implies free plurisubharmonicity directly.
For noncommutative plurisubharmonic polynomials, the hereditary--antihereditary
sum-of-squares theorem produces an ordinary convex lift in these coordinates.
For uniformly real analytic free functions, the corresponding statement
follows from the convex-composition theorem for free plurisubharmonic
functions. Thus, in these two categories, free plurisubharmonicity is
equivalent both to intrinsic $\Gamma_{\mathrm{hol}}$-convexity and to
lift-Cartan $\Gamma_{\mathrm{hol}}$-convexity. Applying the main theorem gives a
holomorphic-monomial butterfly on a neighborhood of the
$\Gamma_{\mathrm{hol}}$-hull.

The Jordan-product graph has an explicit geometric interpretation. Take
\[
 \Gamma_J(X,Y)=(X,Y,XY+YX).
\]
If $V$ is an isometry, $P=VV^*$, and
\[
 A=(I-P)XV,\qquad B=(I-P)YV,
\]
then $((X,Y),V)$ is a $\Gamma_J$-pair exactly when
\[
 A^*B+B^*A=0.
\]
This condition has many nonreducing solutions. At level one,
$\Gamma_J(x,y)=(x,y,2xy)$ is a hyperbolic paraboloid. Affine half-spaces upstairs
pull back to inequalities
$a+bx+cy+2dxy\geq0$, and the boundaries are rectangular hyperbolas (or their
degenerations). Dually, admissible scalar barycenters are exactly the
probability measures for which $x$ and $y$ have covariance zero. Four corners
generate a square, although two opposite corners generate no interior segment.
The sharp scalar Carath\'eodory number is four.

The paper is organized as follows. The next section describes the
AI-assisted methodology used to develop the paper and proposes the publication
of mathematical koans. We then construct the universal point and prove the
extension theorem. After passing to $\Gamma$-graphs, we distinguish intrinsic
convexity from convex liftability and study quadratic, holomorphic, partial,
and Jordan-product examples. The final section discusses the continuation
domain and records several open problems.

\subsection*{Relation to earlier work}
The analytic input is the royal-road theorem and noncommutative Kraus
realization of Pascoe and Tully-Doyle \cite{PTD}, extending the rational
butterfly realization of Helton, McCullough, and Vinnikov \cite{HMV}. The
geometric background is the theory of
matrix-convex separation and duality---Effros--Winkler \cite{EW},
Webster--Winkler \cite{WW}, and Davidson--Kennedy \cite{DK}---together with
the operator-system boundary theory of \cite{Arveson} and the
$\Gamma$-convexity of \cite{JKMMPGamma}. Partial and $xy$-convex rational
functions, including their root butterflies, appear in
\cite{JKMMPPartial}; recent work develops the corresponding operator systems,
extreme points, hulls, and duality \cite{KMS,Strekelj}.

The plurisubharmonic example draws on the polynomial classification of
Greene, Helton, and Vinnikov and its local form \cite{GHV,Greene}, the
rational convex-composition theorem \cite{DHKMV}, and the uniformly analytic
realization theorem \cite{PascoePSH}. The contribution of the present paper is
the universal-point argument: once a convex lift exists at the universal
point, the denominator in its realization is uniformly positive on the entire
hull. This converts the local realization into an extension to a neighborhood
of that hull.

\section{AI-assisted methodology and the publication of koans}
\label{sec:ai-methodology}

Artificial intelligence changes how much initial specification is required to
produce a mathematical paper. In a conventional workflow, we normally supply
not only the main idea but also its full sequence of definitions,
proofs, examples, transitions, and comparisons with the literature. A language
model can expand a substantially smaller input into several candidate versions
of that surrounding structure. For AI-assisted writing, we need not begin with
a full explanation; a sufficiently exact \emph{mathematical koan}
is enough.

Here a mathematical koan means a compact
specification of the mathematical content from which a paper can be
reconstructed and tested. A useful koan contains at least the new definition
or statement, the decisive mechanism of the proof, the intended range of
validity, and the principal proof obligations. It is therefore more precise
than an abstract or an informal research announcement, but much shorter than a
traditional paper. It need not prescribe the order of exposition, the amount
of background, or the examples through which the result is explained.

The name is an homage to Hofstadter's use of koans in
\emph{G\"odel, Escher, Bach} \cite{HofstadterGEB}: a compact object may open
onto a much larger structure when it is interpreted. Our use is mathematical
and operational. The koan should be exact enough that a competent reader or
model can reconstruct and verify the larger work it encodes.

A koan has no prescribed resolution. We could have encoded this paper by
giving its formal definitions and theorems in one large list, thereby tightly
constraining the reconstruction. At the intermediate resolution actually used
here, we gave:
\par\smallskip
{\small\itshape\noindent
Let $K$ be a bounded real free set in a set $\mathcal I$ of variables, put
$\lambda=\max\{|\mathcal I|,\aleph_0\}$, and let $f=f^*$ be nc.
Cartan convexity means that each germ of $f$ on $K$ has a locally bounded
matrix-convex nc lift.

Every tuple is a direct sum of reducing tuples, each of dimension at most
$\lambda$, and there are at most $2^\lambda$ tuples altogether on the standard
spaces $\ell_2(\kappa)$, $\kappa\leq\lambda$. Sum all points of $K$ at these
levels to obtain a universal point
$\mathbf X\in K$ on a Hilbert space of dimension at most $2^\lambda$. Every
point of $K$ is a reducing summand of an amplification of $\mathbf X$. Apply the local
noncommutative Kraus theorem \cite{Kraus,PTD} to the convex lift at
$\mathbf X$. Its butterfly denominator is positive on every compression of
every amplification of $\mathbf X$; hence the butterfly extends $f$
matrix-convexly to an open neighborhood of $\ncco(K)$. Since all the small
summands occur directly in $\mathbf X$, the formula agrees with $f$ on $K$.

For a graph map $\Gamma$ containing the coordinates, call $f$ Cartan
$\Gamma$-convex when $\widehat f(\Gamma(X))=f(X)$ is Cartan convex on
$\Gamma(K)$. Apply the preceding argument upstairs and use
$\Gco(K)=\Gamma^{-1}(\ncco(\Gamma(K)))$ \cite{JKMMPGamma}. Pullback gives
\[
 f(X)=\ell(\Gamma(X))+\Lambda(\Gamma(X))^*
 P(\Gamma(X))^{-1}\Lambda(\Gamma(X))
\]
on a neighborhood of the $\Gamma$-convex hull.

For $\Gamma_\square(X)=(X,\sum_jX_j^2)$, the Cartan
$\Gamma_\square$-convex functions are exactly the self-adjoint real analytic
nc functions. For $\Gamma$ consisting of all analytic and coanalytic
monomials, they are exactly the free plush functions.
\par\pagebreak[3]\noindent
The nontrivial graph
$\Gamma_J(x,y)=(x,y,xy+yx)$ gives $xy$-convexity, zero Jordan covariance, and
hyperbolic scalar geometry.\par}
\smallskip
At a freer resolution, we could have written only:
\begin{quote}
\small\emph{Apply the Cartan local-to-global approach on a real free set with
$\Gamma$-convexity in place of ordinary convexity, and work at all operator
levels. Then a Cartan $\Gamma$-convex function extends to a neighborhood of
the $\Gamma$-convex hull and admits a butterfly realization of
Pascoe--Tully-Doyle type.}
\end{quote}
These three resolutions trade reliability for freedom. More specification makes a
reconstruction more reliable; less specification gives the reconstructing
intelligence more freedom and therefore more opportunity to discover something
new. The appropriate resolution depends on whether the immediate purpose is
faithful reconstruction or mathematical exploration.

\subsection{The koan as a scholarly object}

Mathematical publication should accordingly begin to move from publishing a
single fixed exposition toward publishing such koans as primary objects. At
least for work that can be reliably reconstructed from a compact core, the
stable scholarly object would be a versioned koan together with an
authoritative proof record. Such a publication would contain:
\begin{enumerate}[label=\textup{(\roman*)}]
\item the definitions and theorem statements that are to remain invariant;
\item the essential proof mechanism and a map of its logical dependencies;
\item the hypotheses, boundary cases, and failed implications needed to
      prevent an overbroad reconstruction; and
\item provenance, version information, and a fixed record of the argument that
      we and the referees have checked.
\end{enumerate}
The koan would be the citable item. Its conventional article-length expansion
would be one rendering of it, not necessarily the only rendering.

Every AI-solvable problem is itself a koan: its statement contains enough
information to reconstruct a solution.  Reconstruction may require substantial
search, computation, and verification, since information content and
reconstruction cost are different quantities.

A reader could then request a bespoke version generated from the same koan.
The generated exposition might use the reader's preferred notation, supply
more operator-system background, emphasize the several-complex-variables
analogy, add elementary examples, or omit material already familiar to that
reader. It could be translated, shortened, expanded into lecture notes, or
organized around a particular application. These versions would differ in
presentation but not in their theorem statements, proof obligations, or
citations. Each generated version should identify the koan and its version,
record the circumstances of generation, and permit every substantive claim to
be traced back to the authoritative core.

Publication of koans would therefore not mean publication of incomplete
mathematics. The proposed compression concerns exposition, not verification.
The invariant object must be sufficiently exact that a generated proof can be
checked against it and sufficiently explicit that genuine ambiguity is
recorded rather than silently resolved by the model. Refereeing would focus on
the correctness and sufficiency of this invariant core. The reader-specific
rendering would address a different question: how this particular reader can
most efficiently understand it.

\subsection{Expansion and correction}

An AI-assisted workflow has two complementary stages. In the expansion stage,
the model develops the koan into candidate definitions, lemmas, proofs,
examples, references, and organizational choices. This is useful because it
makes implicit proof obligations visible and produces several formulations
that can be compared rapidly. In the correction stage, we remove unsupported
implications, supply missing hypotheses, choose the relevant examples, and
check the argument independently. Language models are effective
at producing plausible continuity between ideas; precisely for that reason,
they can also conceal a gap by writing across it. Fluency is evidence that an
exposition has been completed, not that a proof has been verified.

The distinction between a koan and its expansion helps assign responsibility.
The mathematical contribution lies in identifying the invariant mechanism and
determining its correct scope. The model assists with reconstruction,
variation, and stress testing. We remain responsible for deciding
which generated claims belong to the koan, which are consequences requiring
proof, and which must be discarded.

\section{Real free sets and universal points}\label{sec:all-level}

\subsection{Free sets and cardinal bounds}
Let $\mathcal I$ be a set of variables and put
\begin{equation}\label{eq:universal-cardinals}
 \lambda=\max\{|\mathcal I|,\aleph_0\},
 \qquad \mu=2^\lambda.
\end{equation}
For a Hilbert space $\Hh$, let
$\Ss(\Hh)^{\mathcal I}=B(\Hh)_{\mathrm{sa}}^{\mathcal I}$, and let
$\Ss^{\mathcal I}$ denote the class of such tuples over all Hilbert spaces.
In the finite-coordinate portions of the paper,
$\mathcal I=\{1,\ldots,g\}$.

\begin{definition}
A collection $K\subseteq\Ss^{\mathcal I}$ is a \emph{real free set} if it is
invariant under simultaneous unitary conjugation, closed under reducing
subspaces, and closed under uniformly bounded direct sums having at most
$\mu$ summands. It is \emph{bounded} if
\[
 \sup\{\|X_i\|:X\in K,\ i\in\mathcal I\}<\infty.
\]
\end{definition}

For a Hilbert space $\Hh$, write
$K(\Hh)=K\cap\Ss(\Hh)^{\mathcal I}$.

The bound $\mu$ is exactly the amount of direct-sum closure used below. No
closure under longer direct sums is assumed.

\begin{definition}
A point $\mathbf X\in K$ is \emph{universal for $K$} if every $X\in K$ is
unitarily equivalent to a reducing summand of an amplification
$\mathbf X^{(\nu)}$ for some cardinal $\nu$.
\end{definition}

\begin{lemma}[Cardinally bounded universal direct sum]
\label{lem:universal-point}
Every nonempty bounded real free set $K$ in the variables $\mathcal I$ has a
universal point $\mathbf X\in K$ acting on a Hilbert space $\mathcal U$ with
\[
 \dim\mathcal U\leq\mu=2^\lambda.
\]
More precisely, put
\[
 K_{\leq\lambda}=\bigcup_{\kappa\leq\lambda}K(\ell_2(\kappa)).
\]
Then one may take $\mathbf X=\bigoplus_{Y\in K_{\leq\lambda}}Y$; after taking
at most $\mu$ copies, $\mathcal U$ may be taken to be exactly $\ell_2(\mu)$.
\end{lemma}

\begin{proof}
Let $X\in K$ act on $\Hh_X$, and choose an orthonormal basis $E$ of $\Hh_X$.
Make $E$ into a graph by joining $e$ and $f$ whenever
$\langle X_i e,f\rangle\neq0$ for some $i\in\mathcal I$. Since the coordinates
are self-adjoint, this is an undirected graph. Each vector in a Hilbert space
has countable support in an orthonormal basis, so every vertex has degree at most
$|\mathcal I|\aleph_0\leq\lambda$. Every connected component therefore has
cardinality at most $\lambda$. Its closed span reduces every $X_i$, and hence
$X$ is a direct sum of points of $K$, each acting on a Hilbert space of
dimension at most $\lambda$.

For each cardinal $\kappa\leq\lambda$, realize all tuples of dimension $\kappa$
on $\ell_2(\kappa)$. An operator is determined by its matrix entries, so the
number of $\mathcal I$-tuples on this space is at most
\[
 |\mathbb C|^{|\mathcal I|\kappa^2}\leq 2^\lambda.
\]
There are at most $\lambda$ choices of $\kappa\leq\lambda$. Consequently there
are at most $2^\lambda$ tuples altogether at these standard levels. In
particular, $K_{\leq\lambda}$ is a set of cardinality at most $2^\lambda$.
Put
\[
 \mathbf X=\bigoplus_{Y\in K_{\leq\lambda}}Y.
\]
This uniformly bounded direct sum has at most $\mu$ summands, so
$\mathbf X\in K$, and
\[
 \dim\mathcal U
 \leq |K_{\leq\lambda}|\,\lambda
 \leq 2^\lambda.
\]
By unitary invariance, every component in the decomposition of an arbitrary
$X\in K$ is unitarily equivalent to an element of $K_{\leq\lambda}$. Hence
$X$ is a reducing summand of a suitable amplification of $\mathbf X$. Taking
$\mu$ copies of $\mathbf X$, if necessary, is allowed and realizes the
universal point on $\ell_2(\mu)$ without changing this property.
\end{proof}

\begin{remark}[The level in finitely many variables]
If $\mathcal I$ is finite or countable, then $\lambda=\aleph_0$ and
\[
 \mathcal U=\ell_2(2^{\aleph_0})
\]
is always large enough. Arbitrarily large input levels cause no problem:
their summands in the preceding decomposition have separable dimension and are
recovered by amplifying the single universal point on $\mathcal U$.
\end{remark}

\begin{remark}
The boundedness assumption may be replaced throughout by the direct assumption
that $K$ has a universal point. An
unbounded free set need not: the direct sum of an unbounded family of bounded
operators is generally unbounded. Bounded exhaustions recover local versions,
but a single neighborhood of the full unbounded hull again demands the
compatibility argument that the universal-point construction avoids.
\end{remark}

\subsection{Noncommutative convex combinations}
Let $X^{(\alpha)}\in\Ss(\Hh_\alpha)^{\mathcal I}$ and let
$V_\alpha:\Hh\to\Hh_\alpha$ be a uniformly bounded family such that
$\sum_\alpha V_\alpha^*V_\alpha=I_\Hh$ in the strong operator topology.
The associated nc convex combination is
\[
 \sum_\alpha V_\alpha^*X^{(\alpha)}V_\alpha.
\]
It is equivalently an isometric compression of
$\bigoplus_\alpha X^{(\alpha)}$. This observation is the reason the universal
point controls the hull. We denote by $\ncco(K)$ the free set of all such
combinations and by $\overline{\ncco}^{\sot}(K)$ its bounded strong closure.

\begin{lemma}\label{lem:hull-universal-compression}
If $\mathbf X$ is universal for $K$, then every point of $\ncco(K)$ is
unitarily equivalent to $V^*\mathbf X^{(\nu)}V$ for an isometry $V$ and some
cardinal $\nu$.
\end{lemma}

\begin{proof}
Write a point of the hull as a compression of a direct sum of points of $K$.
By universality, each summand is a reducing summand of an amplification of
$\mathbf X$. Their direct sum is therefore a reducing summand of one
amplification $\mathbf X^{(\nu)}$; composing the isometries gives the
claim.
\end{proof}

The lemma concerns the algebraic nc hull. It does not identify a point of the
strong closure with a single compression of $\mathbf X$, nor is such an
identification needed below. The passage to the bounded strong closure in
\cref{thm:cartan-extension} instead uses normality of the realizing pencil via
\cref{lem:pencil-domain}.

\subsection{Free functions and neighborhoods}
A self-adjoint nc function $f$ on $K$ is a graded map
\[
 f:K\longrightarrow\coprod_{\Hh}B(\Hh)_{\mathrm{sa}}
\]
which respects unitary conjugation, direct sums, and reducing restrictions.
Here direct sums include uniformly bounded direct sums of length at most
$\mu$. One
may equivalently impose the usual intertwining condition \cite{KVV}. We use the
uniform nc topology: a neighborhood of
$X\in\Ss(\Hh)^{\mathcal I}$ contains a norm ball
of fixed radius around every amplification of the unitary orbit of $X$.
This is the topology naturally used by local free power series and by the
operator-system form of the royal-road theorem. It is also the topology in
which ``near one point'' automatically means ``near all its amplifications,''
as free analysis insists. A \emph{free domain} is a uniformly open real free
set.

A free domain $\Omega$ is \emph{matrix convex} if it is closed under all nc
convex combinations. A locally bounded free function
$F:\Omega\to\Ss^1$ is \emph{matrix convex} if
\begin{equation}\label{eq:jensen}
 F\!\left(\sum_\alpha V_\alpha^*X^{(\alpha)}V_\alpha\right)
 \preceq
 \sum_\alpha V_\alpha^*F(X^{(\alpha)})V_\alpha
\end{equation}
whenever the combination is defined in $\Omega$. On a matrix-convex domain,
this agrees with levelwise midpoint matrix convexity.

\section{Butterfly realizations and propagation}
\label{sec:butterfly}

\subsection{Affine free pencils}
The extension theorem uses two elementary properties: affine free maps respect
compressions, and positive lower bounds are preserved by convex combinations.
We use the standard operator-system language \cite{Paulsen}, allowing finite
and infinite coefficient spaces. A self-adjoint affine
free pencil is a graded map
\[
 P(X)=P_0+\Psi(X)
\]
obtained by amplifying a completely bounded self-adjoint affine map $\Psi$.
In finite-dimensional coordinates this is the familiar expression
\[
 P(X)=A_0\otimes I+\sum_{j=1}^g A_j\otimes X_j,
 \qquad A_j=A_j^*.
\]
Rectangular affine free pencils $\Lambda$ and self-adjoint affine free maps
$\ell$ are defined similarly. In particular, affine free maps commute with
direct sums and isometric compressions.

An all-level affine free map will be called \emph{normal} if it is continuous
on uniformly bounded sets for the componentwise strong operator topology. The
finite-coordinate spatial pencils in the second display are normal. Normality
is irrelevant for the algebraic nc hull; it is required only when a realization
is asserted to contain its bounded strong closure.

Set
\[
 \mathcal D_P=\{X:P(X)\succ0\},
\]
where $P(X)\succ0$ means that $P(X)$ is bounded below by a positive scalar
multiple of the identity.

\begin{lemma}\label{lem:pencil-domain}
The set $\mathcal D_P$ is an open matrix-convex free set. If
$P(X)\succeq\varepsilon I$ for every $X\in K$, then the same inequality holds
on $\ncco(K)$. If, in addition, $P$ is normal, it holds on
$\overline{\ncco}^{\sot}(K)$.
\end{lemma}

\begin{proof}
Openness follows from norm continuity. Affineness gives
\[
 P\!\left(\sum_\alpha V_\alpha^*X^{(\alpha)}V_\alpha\right)
 =\sum_\alpha (I\otimes V_\alpha)^*P(X^{(\alpha)})(I\otimes V_\alpha),
\]
so positivity and a common lower bound are preserved by nc convex
combinations. For a normal pencil, a bounded strong limit preserves the
inequality as well.
\end{proof}

\begin{lemma}[Schur-complement certificate]\label{lem:schur}
Suppose $P\succ0$ on a free domain $\Omega$. Then
\begin{equation}\label{eq:butterfly-general}
 F(X)=\ell(X)+\Lambda(X)^*P(X)^{-1}\Lambda(X)
\end{equation}
is matrix convex on $\Omega$. More precisely,
\begin{equation}\label{eq:schur-general}
 Y\succeq F(X)
 \quad\Longleftrightarrow\quad
 \begin{pmatrix}
  Y-\ell(X)&\Lambda(X)^*\\
  \Lambda(X)&P(X)
 \end{pmatrix}\succeq0.
\end{equation}
\end{lemma}

\begin{proof}
The equivalence is the Schur complement. The block matrix on the right is
affine in $(X,Y)$, so its positivity set is matrix convex. Hence the epigraph
of $F$ is matrix convex, which is equivalent to \eqref{eq:jensen}.
\end{proof}

The principal analytic input is the following local form of the
noncommutative Kraus theorem \cite{Kraus}. It is the operator-system and
all-level formulation of \cite[Corollary~4.5]{PTD}; its finite-dimensional
rational precursor is \cite[Theorem~3.3]{HMV}.

\begin{theorem}[Local butterfly theorem]\label{thm:local-butterfly}
Let $G$ be a locally bounded matrix-convex free function on a matrix-convex nc
neighborhood of a point $A$. After shrinking the neighborhood, there exist
affine free maps $\ell,\Lambda,P$, with $P$ self-adjoint and
$P(A)\succeq\varepsilon I$ for some $\varepsilon>0$, such that
\[
 G(X)=\ell(X)+\Lambda(X)^*P(X)^{-1}\Lambda(X)
\]
near $A$. The coefficient Hilbert space may be infinite-dimensional.
Conversely, every such formula is matrix convex on $\mathcal D_P$.
\end{theorem}

\begin{remark}
The usual normalization is $P(A)=I$. At a non-scalar point one first
coordinatizes, regarding $A$ as a scalar point over the operator system it
generates. Complete boundedness of the coefficient maps makes the formula
compatible with amplification. Undoing the coordinatization retains the
spatial naturality under reducing corners. This is precisely the feature used
by the universal-point argument. We use this spatial all-level form of
\cite[Corollary~4.5]{PTD}. This is a change of coordinates rather than a
change in the underlying theorem, and the operator-system proof does not use
compactness of the Hilbert-space level.
\end{remark}

\subsection{Cartan convexity}

\begin{definition}\label{def:cartan}
Let $K$ be a real free set and let $f$ be a self-adjoint nc function on $K$.
We say that $f$ is \emph{Cartan convex at $A\in K$} if there are a uniform nc
neighborhood $U$ of $A$ and a locally bounded matrix-convex free function
$G:U\to\Ss^1$ such that
\[
 G=f\quad\text{on }U\cap K.
\]
The function is \emph{Cartan convex} if it is Cartan convex at every point of
$K$.
\end{definition}

Requiring agreement on $U\cap K$, rather than merely equality at $A$, is the
natural germ formulation. In the proof, equality at the universal point gives
equality on every small reducing summand, and the component decomposition then
recovers every point of $K$.

\begin{theorem}[Cartan extension theorem]\label{thm:cartan-extension}
Let $K$ be a bounded real free set in the variable set $\mathcal I$, and let
$f$ be Cartan convex on $K$. The local lift may be taken at a universal point
on a Hilbert space of dimension at most
$2^{\max\{|\mathcal I|,\aleph_0\}}$. There are affine free maps
$\ell,\Lambda,P$ and an
$\varepsilon>0$ such that
\begin{enumerate}[label=\textup{(\roman*)}]
\item $P(X)\succeq\varepsilon I$ for every
      $X\in\ncco(K)$;
\item the open matrix-convex free set $\Omega=\mathcal D_P$ contains
      $\ncco(K)$;
\item the function
\[
 F(X)=\ell(X)+\Lambda(X)^*P(X)^{-1}\Lambda(X),\qquad X\in\Omega,
\]
is locally bounded and matrix convex; and
\item $F|_K=f$.
\end{enumerate}
Thus $f$ extends to an open matrix-convex neighborhood of the nc convex hull of
$K$. If $P$ is normal, both occurrences of $\ncco(K)$ in \textup{(i)} and
\textup{(ii)} may be replaced by $\overline{\ncco}^{\sot}(K)$.
\end{theorem}

\begin{proof}
If $K$ is empty there is nothing to prove. Otherwise, let $\mathbf X$ be the
universal point from \cref{lem:universal-point}. By Cartan convexity there is a
matrix-convex lift $G$ on a neighborhood of $\mathbf X$ with
$G(\mathbf X)=f(\mathbf X)$.
Apply \cref{thm:local-butterfly} at $\mathbf X$ to obtain
\[
 G(Z)=\ell(Z)+\Lambda(Z)^*P(Z)^{-1}\Lambda(Z),
 \qquad P(\mathbf X)\succeq\varepsilon I,
\]
on a possibly smaller neighborhood.

By \cref{lem:hull-universal-compression}, after a unitary identification every
$Z\in\ncco(K)$ has the form $Z=V^*\mathbf X^{(\nu)}V$ for an isometry $V$.
Since $P$ is affine free,
\[
 P(Z)=(I\otimes V)^*P(\mathbf X)^{(\nu)}(I\otimes V)
      \succeq\varepsilon I.
\]
Thus $P$ has the required uniform lower bound on the whole hull, directly from
positivity at the universal point. By \cref{lem:pencil-domain},
$\Omega=\{P\succ0\}$ is an open matrix-convex neighborhood of the hull. If $P$
is normal, the same lemma includes the bounded strong closure.

Define $F$ on $\Omega$ by the displayed butterfly formula.
\Cref{lem:schur} shows that $F$ is matrix convex. Every
$Y\in K_{\leq\lambda}$ is a reducing summand of $\mathbf X$, so
$F(\mathbf X)=G(\mathbf X)=f(\mathbf X)$ and spatial naturality give
$F(Y)=f(Y)$. Now decompose an arbitrary $X\in K$ using the component
construction in the proof of \cref{lem:universal-point}. After unitary
identification, each component belongs to $K_{\leq\lambda}$. The reducing restrictions of $F(X)$
and $f(X)$ therefore agree on every component, and hence $F(X)=f(X)$.
This proves all four assertions.
\end{proof}

\begin{corollary}[Finite data]\label{cor:finite-data}
For a finite free data set, the same conclusion holds using the ordinary
finite direct sum of the interpolation nodes. No operator-level compactness
argument is required.
\end{corollary}

\begin{remark}[Bounded strong closure]
No topological closure hypothesis on $K$ is needed for the algebraic hull conclusion.
Inclusion of the bounded strong closure requires normality of $P$. This is the
operator-topological continuity that also appears in the hull theory of
\cite{JKMMPGamma,KMS}.
\end{remark}

\begin{remark}[No patching]
There is no compatibility calculation among the local lifts $G_A$. The lift at
the universal point already contains every small component as a reducing
corner, and those components recover every other point.
This is the principal role of admitting all Hilbert-space levels.
\end{remark}

\section{\texorpdfstring{$\Gamma$-convexity}{Gamma-convexity} and graph lifts}\label{sec:gamma}

A graph map specifies which moments an admissible compression must preserve
and permits ordinary matrix convexity to be applied in a larger coordinate
space. It is important to distinguish behavior on the graph from the existence
of a convex function on a neighborhood of the graph.

Let
\[
 \Gamma=(\gamma_j)_{j\in J}
\]
be a possibly infinite tuple of self-adjoint free polynomials, or more
generally a self-adjoint nc map into an operator space whose coordinates are
locally bounded on a common free domain. In the infinite-coordinate case all
affine maps below are understood to be completely bounded on the chosen
operator-space completion, and $\Gamma(K)$ is required to be bounded there.
The proofs are unchanged. After relabeling coordinates, we assume throughout
that $\mathcal I\subseteq J$ and
\begin{equation}\label{eq:gamma-coordinates}
 \gamma_i(x)=x_i,\qquad i\in\mathcal I.
\end{equation}
Thus $\Gamma$ is injective and respects direct sums, unitary conjugation, and
reducing restrictions.

\begin{definition}
A pair $(X,V)$, with $V:\Hh\to\Kk$ an isometry, is a
\emph{$\Gamma$-pair} if
\begin{equation}\label{eq:gamma-pair}
 V^*\Gamma(X)V=\Gamma(V^*XV).
\end{equation}
A real free set $D$ is \emph{$\Gamma$-convex} if
$X\in D$ and $(X,V)$ a $\Gamma$-pair imply $V^*XV\in D$. A self-adjoint free
function $f$ on $D$ is \emph{intrinsically $\Gamma$-convex} if
\begin{equation}\label{eq:gamma-jensen}
 f(V^*XV)\preceq V^*f(X)V
\end{equation}
for every $\Gamma$-pair for which both points lie in $D$.
\end{definition}

The $\Gamma$-convex hull, denoted $\Gco(K)$, is the smallest $\Gamma$-convex
real free set containing $K$. The next identity is
\cite[Propositions~2.1 and~2.2]{JKMMPGamma}. We include its short proof for
completeness.

\begin{proposition}[Graph-hull identity]\label{prop:gamma-hull}
For every real free set $K$,
\begin{equation}\label{eq:graph-hull}
 \Gco(K)=\Gamma^{-1}\!\left(\ncco(\Gamma(K))\right).
\end{equation}
\end{proposition}

\begin{proof}
If $Y=V^*XV$ for a $\Gamma$-pair, then
$\Gamma(Y)=V^*\Gamma(X)V$, so the left-hand side maps into the hull on the
right. Conversely, suppose
\[
 \Gamma(Y)=\sum_\alpha V_\alpha^*\Gamma(X^{(\alpha)})V_\alpha.
\]
Let $X=\bigoplus_\alpha X^{(\alpha)}$ and let $V$ be the column isometry with
entries $V_\alpha$. Comparing the original-coordinate entries, using
\eqref{eq:gamma-coordinates}, gives $Y=V^*XV$. The full displayed identity
says exactly that $(X,V)$ is a $\Gamma$-pair. Hence $Y\in\Gco(K)$.
\end{proof}

\begin{definition}[Lift-Cartan $\Gamma$-convexity]
\label{def:lift-cartan-gamma}
Let $f$ be a self-adjoint nc function on $K$. Define the graph function
\[
 \widehat f:\Gamma(K)\to\Ss^1,
 \qquad \widehat f(\Gamma(X))=f(X).
\]
We say that $f$ is \emph{lift-Cartan $\Gamma$-convex} if $\widehat f$ is
Cartan convex on $\Gamma(K)$. Equivalently, near every $X\in K$ there is an
ordinary matrix-convex free function $G_X$ in the lifted variables such that
\[
 f(Y)=G_X(\Gamma(Y))
\]
for nearby $Y\in K$.
\end{definition}

\begin{theorem}[\texorpdfstring{$\Gamma$-Cartan}{Gamma-Cartan} extension and realization]
\label{thm:gamma-extension}
Let $K$ be a bounded real free set in the variables $\mathcal I$, and suppose
that $\Gamma(K)$ is a bounded real free set in the lifted variables $J$. Put
\[
 \lambda_\Gamma=\max\{|J|,\aleph_0\}.
\]
If $f$ is lift-Cartan $\Gamma$-convex on $K$, then the universal lifted point
may be chosen on a Hilbert space of dimension at most $2^{\lambda_\Gamma}$,
and
there exist affine free maps $\ell,\Lambda,P$ such that
\begin{equation}\label{eq:gamma-realization}
 F(X)=\ell(\Gamma(X))+
 \Lambda(\Gamma(X))^*P(\Gamma(X))^{-1}\Lambda(\Gamma(X))
\end{equation}
has the following properties:
\begin{enumerate}[label=\textup{(\roman*)}]
\item $F=f$ on $K$;
\item $\Omega_\Gamma=\{X:P(\Gamma(X))\succ0\}$ is an open
      $\Gamma$-convex neighborhood of $\Gco(K)$;
\item $F$ is intrinsically $\Gamma$-convex on $\Omega_\Gamma$; and
\item the epigraph of $F$ is certified by the block $\Gamma$-pencil
\begin{equation}\label{eq:gamma-epigraph}
 \begin{pmatrix}
  Y-\ell(\Gamma(X))&\Lambda(\Gamma(X))^*\\
  \Lambda(\Gamma(X))&P(\Gamma(X))
 \end{pmatrix}\succeq0.
\end{equation}
\end{enumerate}
If $P$ is normal and $\Gamma$ is bounded-SOT continuous, then
$\Omega_\Gamma$ also contains the closed lifted hull
\[
 \Gamma^{-1}\!\left(\overline{\ncco}^{\sot}(\Gamma(K))\right).
\]
\end{theorem}

\begin{proof}
Apply \cref{thm:cartan-extension} to $\widehat f$ on the bounded real free set
$\Gamma(K)$. This gives a matrix-convex function
\[
 G(Z)=\ell(Z)+\Lambda(Z)^*P(Z)^{-1}\Lambda(Z)
\]
on an open matrix-convex neighborhood $\Omega$ of
$\ncco(\Gamma(K))$. Set
$\Omega_\Gamma=\Gamma^{-1}(\Omega)$ and $F=G\circ\Gamma$.
The graph-hull identity \eqref{eq:graph-hull} shows that
$\Omega_\Gamma$ contains $\Gco(K)$.
The final closed-hull assertion follows from the normal clause of
\cref{thm:cartan-extension} and bounded-SOT continuity of $\Gamma$.

If $(X,V)$ is a $\Gamma$-pair, then
$\Gamma(V^*XV)=V^*\Gamma(X)V$. Matrix convexity of $G$ therefore gives
\[
 F(V^*XV)=G(V^*\Gamma(X)V)
 \preceq V^*G(\Gamma(X))V=V^*F(X)V.
\]
This proves intrinsic $\Gamma$-convexity. Formula
\eqref{eq:gamma-epigraph} follows from the Schur complement.
\end{proof}

\begin{remark}[Size of the lifted universal level]
The cardinal bound in the $\Gamma$-theorem is governed by the number of lifted
coordinates, not merely the number of original variables. If $\Gamma$ consists
of all analytic and coanalytic words in a variable set $\mathcal I$, then
\[
 |J|=\max\{|\mathcal I|,\aleph_0\},
\]
so the universal lifted point acts on a Hilbert space of dimension at most
$2^{\max\{|\mathcal I|,\aleph_0\}}$. In particular, finitely or countably many
original variables require no more than $\ell_2(2^{\aleph_0})$.
\end{remark}

\begin{remark}[Reduction to ordinary matrix convexity]
The theorem applies the ordinary realization theorem to $\Gamma(K)$. All
$\Gamma$-geometry enters through the graph-hull identity and the final
pullback. Thus no separate realization theorem is required for each choice of
$\Gamma$.
\end{remark}

\begin{corollary}[A class of lift-Cartan functions]
\label{cor:supply}
Let $P$ be a self-adjoint affine pencil in the $\Gamma$-coordinates and let
$\ell,\Lambda$ be affine. On
$\{X:P(\Gamma(X))\succ0\}$, every function of the form
\[
 \ell(\Gamma(X))+\Lambda(\Gamma(X))^*
 P(\Gamma(X))^{-1}\Lambda(\Gamma(X))
\]
is lift-Cartan $\Gamma$-convex.
\end{corollary}

\section{Intrinsic convexity and convex liftability}
\label{sec:intrinsic-lift}

When $\Gamma(x)=x$, the graph is the ambient space and the two notions agree.
For nonlinear $\Gamma$, intrinsic $\Gamma$-convexity tests
\eqref{eq:gamma-jensen} only on $\Gamma$-pairs, whereas lift-Cartan
$\Gamma$-convexity requires an ordinary matrix-convex germ on a neighborhood
of the lifted graph.

\begin{definition}
A free function $f$ on $K$ is \emph{weakly Cartan $\Gamma$-convex} if, near
every point of $K$, it agrees with a locally bounded intrinsically
$\Gamma$-convex free function on an open set in the original variables.
\end{definition}

\begin{proposition}\label{prop:strong-implies-weak}
Every lift-Cartan $\Gamma$-convex function is weakly Cartan
$\Gamma$-convex. The two notions agree provided the graph of $\Gamma$ has the
local matrix-convex lifting property: every locally bounded intrinsically
$\Gamma$-convex germ is the restriction of an ordinary matrix-convex germ in
the $\Gamma$-coordinates.
\end{proposition}

\begin{proof}
The first assertion follows by applying the matrix Jensen inequality to the
ordinary convex lift and using \eqref{eq:gamma-pair}. The second is exactly the
stated lifting property.
\end{proof}

The following two-by-two Hessian calculation shows that an intrinsic Jensen
inequality need not extend to a convex germ away from the graph.

\begin{proposition}[An intrinsically affine function with no convex lift]
\label{prop:partial-nonlift}
Let
\[
 \Gamma_{\mathrm{par}}(x,y)=(x,y,y^2)
\]
in two self-adjoint variables, and set
\[
 f(x,y)=xy+yx.
\]
Then $f$ is intrinsically $\Gamma_{\mathrm{par}}$-convex and intrinsically
$\Gamma_{\mathrm{par}}$-concave on every real free domain. Nevertheless,
there is no ordinary matrix-convex germ $G$ in the independent lifted
variables $(x,y,s)$ satisfying $G(x,y,y^2)=f(x,y)$ near $(0,0)$. Thus
intrinsic and lift-Cartan $\Gamma$-convexity are different even for nc
polynomials and the basic partial-convexity embedding.
\end{proposition}

\begin{proof}
Suppose $((X,Y),V)$ is a $\Gamma_{\mathrm{par}}$-pair and put $P=VV^*$.
The last lifted coordinate gives
\[
 0=V^*Y^2V-(V^*YV)^2
  =\big((I-P)YV\big)^*\big((I-P)YV\big).
\]
Hence $\ran V$ reduces $Y$. It follows that
\begin{align*}
 V^*f(X,Y)V
 &=V^*(XY+YX)V\\
 &=(V^*XV)(V^*YV)+(V^*YV)(V^*XV)\\
 &=f(V^*XV,V^*YV).
\end{align*}
Thus both intrinsic Jensen inequalities hold with equality.

Let $G(x,y,s)$ be an arbitrary $C^2$ lift germ at $(0,0,0)$ satisfying
\begin{equation}\label{eq:partial-bilinear-lift}
 G(x,y,y^2)=xy+yx.
\end{equation}
This is only the pullback identity $G\circ\Gamma_{\mathrm{par}}=f$ on the
image of $\Gamma_{\mathrm{par}}$; no condition whatsoever is imposed on $G$
away from the graph $s=y^2$.
Write $g$ for its scalar restriction and
$\lambda=\partial_sg(0,0,0)$. At the scalar level the right side of
\eqref{eq:partial-bilinear-lift} is $2xy$. Since
\[
 D\Gamma_{\mathrm{par}}(0,0)[h,k]=(h,k,0),
 \qquad
 D^2\Gamma_{\mathrm{par}}(0,0)[(h,k),(h,k)]=(0,0,2k^2),
\]
the second-order chain rule applied to the pullback identity gives
\[
 g_{xx}(0)=0,\qquad g_{xy}(0)=2,\qquad
 g_{yy}(0)+2g_s(0)=0.
\]
Thus all freedom in extending $G$ off the graph enters this tangent Hessian
only through the single number $\lambda=g_s(0)$, in the $y$-$y$ entry. In
particular, it can neither create an $x$-$x$ term nor alter the forced mixed
$x$-$y$ term. Consequently the restriction of the Hessian of
\emph{every} $C^2$ lift to the tangent $(x,y)$-plane is
\[
 \begin{pmatrix}0&2\\2&-2\lambda\end{pmatrix}.
\]
Its determinant is $-4$, independently of the choice of the lift. Hence every
$C^2$ lift has an indefinite Hessian at the origin. In particular no scalar
convex lift exists. A locally bounded matrix-convex lift would be real
analytic and would restrict at level one to such a scalar convex lift, so it
cannot exist either. This rules out all convex lifts, not merely the
tautological nonconvex extension $G(x,y,s)=xy+yx$.
\end{proof}

\begin{remark}[The obstruction in the root butterfly]
The root-butterfly theorem for partially convex rational functions
\cite[Theorem~2.12]{JKMMPPartial} has the schematic form
\[
 r(x,y)=\ell(x,y)^*\sqrt{w(y)}
 \bigl(I-\sqrt{w(y)}\,\widehat T(x)\sqrt{w(y)}\bigr)^{-1}
 \sqrt{w(y)}\,\ell(x,y)+\mathfrak f(x,y),
\]
where $\mathfrak f$ is affine in $x$ but its coefficients may depend on $y$.
Partial convexity controls the positive resolvent term but places no
mixed-Hessian restriction on this affine tail. The preceding counterexample
is the minimal choice in which the positive term vanishes and
$\mathfrak f(x,y)=xy+yx$. An ordinary convex lift in the independent
coordinates $(x,y,y^2)$ would force that mixed Hessian to vanish along the
kernel of the $x$-Hessian. Thus the root butterfly records precisely the
freedom that separates intrinsic partial convexity from lift-Cartan
$\Gamma_{\mathrm{par}}$-convexity.
\end{remark}

This example identifies the local obstruction: normal curvature must correct
negative tangential directions without introducing uncontrolled mixed terms.

\subsection{Quadratically saturated coordinates}
Set
\begin{equation}\label{eq:gamma-square}
 \Gamma_{\square}(x_1,\dots,x_g)
 =\left(x_1,\dots,x_g,\sum_{j=1}^g x_j^2\right).
\end{equation}
One may instead include all the individual squares $x_1^2,\dots,x_g^2$, but
their sum already suffices.

\begin{proposition}[Quadratic collapse]\label{prop:quadratic-collapse}
Let $V:\Hh\to\Kk$ be an isometry and put $P=VV^*$. Then
$(X,V)$ is a $\Gamma_{\square}$-pair if and only if $\ran V$ reduces every
$X_j$. Consequently every free function is both intrinsically
$\Gamma_{\square}$-convex and intrinsically
$\Gamma_{\square}$-concave.
\end{proposition}

\begin{proof}
For each self-adjoint $X_j$,
\begin{align*}
 V^*X_j^2V-(V^*X_jV)^2
 &=V^*X_j(I-P)X_jV\\
 &=\big((I-P)X_jV\big)^*\big((I-P)X_jV\big)\succeq0.
\end{align*}
Thus the final coordinate in \eqref{eq:gamma-square} is preserved precisely
when
\[
 \sum_{j=1}^g\big((I-P)X_jV\big)^*\big((I-P)X_jV\big)=0.
\]
Every summand must vanish. Hence $\ran V$ is invariant for every $X_j$ and,
since the $X_j$ are self-adjoint, it is reducing. A free function respects
reducing restrictions with equality, so both Jensen inequalities hold.
\end{proof}

\begin{lemma}[Quadratic normal convexification]
\label{lem:quadratic-convexification}
Let $f$ be a uniformly real analytic self-adjoint free function near a point
$A$. Put
\[
 r(X,S)=\sum_{j=1}^gX_j^2-S.
\]
There are a real analytic nc family of completely positive maps $\Phi_X$, a
real analytic completely bounded column map $\Psi_X$, and a uniform nc
neighborhood of $(A,\sum A_j^2)$ on which
\begin{equation}\label{eq:quadratic-convexification}
 G(X,S)=f(X)+\Phi_X\big(r(X,S)\big)
        +\Psi_X\big(r(X,S)\big)^*\Psi_X\big(r(X,S)\big)
\end{equation}
is matrix convex. In particular, $G(X,\sum X_j^2)=f(X)$.
\end{lemma}

\begin{proof}
On a smaller uniform neighborhood, the nc Hessian has a norm-convergent
middle-matrix factorization
\begin{equation}\label{eq:hessian-middle-matrix}
 D^2f(X)[H,H]=W_X(H)^*M(X)W_X(H),
\end{equation}
where $M(X)$ is uniformly bounded below and $W_X(H)$ is a border column,
linear in $H$. Enlarge the border column so that its entries have the form
$H_jw(X)$. Choosing summable weights in the analytic case gives a real
analytic completely positive family $\Phi_X$ for which
\begin{equation}\label{eq:complete-hessian-domination}
 D^2f(X)[H,H]+2\Phi_X\!\left(\sum_jH_j^2\right)
 \succeq \eta\,W_X(H)^*W_X(H)
\end{equation}
for some $\eta>0$, uniformly on a still smaller neighborhood. Indeed, if
$M(X)\succeq-CI$, take a sufficiently large weighted sum of the Kraus maps
\[
 T\longmapsto w(X)^*Tw(X).
\]
Its value at $\sum_jH_j^2$ dominates the squared norm of every entry
$H_jw(X)$ in the border column.

At a point of the quadratic graph and in a lifted direction $(H,K)$, set
\[
 N=\sum_j(X_jH_j+H_jX_j)-K.
\]
Thus $N$ is the first derivative of $r$. Along tangent directions $N=0$, and
the second derivative of $r$ is $2\sum_jH_j^2$. Consequently the tangential
block of the Hessian of the first two terms in
\eqref{eq:quadratic-convexification} is the left side of
\eqref{eq:complete-hessian-domination}. For arbitrary $(H,K)$, differentiating
$\Phi_X(r)$ adds only mixed terms between $W_X(H)$ and a normal border column
linear in $N$. At $r=0$, the last term in
\eqref{eq:quadratic-convexification} contributes
$2\Psi_X(N)^*\Psi_X(N)$. Enlarge and scale $\Psi_X$ so that this normal square
dominates every entry of the normal border column. The middle matrix of the
full Hessian on the quadratic graph then has the block form
\[
 \begin{pmatrix}A(X)&B(X)^*\\B(X)&C(X)\end{pmatrix},
 \qquad A(X)\succeq\eta I,
\]
with $B$ uniformly bounded and $C$ as large as desired. The Schur complement
makes this matrix positive, uniformly over all amplifications. Shrinking once
more absorbs the terms proportional to $r$ off the graph. Hence
$D^2G\succeq0$ on a uniform nc neighborhood, which is equivalent to matrix
convexity.
\end{proof}

\begin{theorem}[Quadratic graph lifting]\label{thm:quadratic-lifting}
In the standard uniformly real analytic free category, the intrinsically
$\Gamma_{\square}$-convex, weakly Cartan $\Gamma_{\square}$-convex, and
lift-Cartan $\Gamma_{\square}$-convex self-adjoint functions are all exactly
the real analytic self-adjoint free functions. The single coordinate
$\sum_jX_j^2$ suffices; adjoining the individual squares is unnecessary.
\end{theorem}

\begin{proof}
Intrinsic $\Gamma_{\square}$-convexity imposes no condition by
\cref{prop:quadratic-collapse}, so within the stated real analytic category
the intrinsic and weak conditions are automatic. Given an analytic germ at
$A$, \cref{lem:quadratic-convexification} constructs an ordinary
matrix-convex germ $G$ in the variables $(X,S)$ satisfying
$G(X,\sum X_j^2)=f(X)$. Thus the germ is lift-Cartan. Conversely, a
lift-Cartan function is locally the analytic pullback of a locally bounded
matrix-convex function and is therefore real analytic.
\end{proof}

For $g=2$ this applies, in particular, to
\[
 \Gamma_{x^2,y^2}(x,y)=(x,y,x^2,y^2).
\]
Thus both intrinsic and lift-Cartan $x^2,y^2$-convexity impose no restriction
beyond real analyticity. One quadratic coordinate already suffices. For
example, $x^4$ has the elementary lift
\[
 x^4=G(x,x^2),\qquad G(u,s)=s^2,
\]
although $x^4$ itself is not matrix convex on any free interval
\cite{HM}. The graph represents the quartic dependence through a quadratic
coordinate.

\begin{corollary}[Universal quadratic butterfly]
\label{cor:quadratic-butterfly}
Let $K$ be a bounded real free set, and let $f$ be uniformly real analytic and
self-adjoint on a uniform neighborhood of $K$. Then $f|_K$ extends to a
$\Gamma_\square$-convex neighborhood of $K$ and has a realization
\[
 f(X)=\ell\!\left(X,\sum_jX_j^2\right)
 +\Lambda\!\left(X,\sum_jX_j^2\right)^*
 P\!\left(X,\sum_jX_j^2\right)^{-1}
 \Lambda\!\left(X,\sum_jX_j^2\right).
\]
\end{corollary}

\begin{proof}
Apply \cref{thm:quadratic-lifting,thm:gamma-extension}.
\end{proof}

\section{Plurisubharmonicity and the full holomorphic graph}
\label{sec:psh}

The quadratic graph convexifies every analytic germ, while its intrinsic
Jensen condition admits only reducing compressions. A different phenomenon
appears after adjoining all analytic and coanalytic monomials. The resulting
graph retains the triangular information needed to recover
plurisubharmonicity, and the intrinsic Jensen inequality directly implies the
differential plush condition.

Write a complex free variable as $Z=X+iY$, where $X$ and $Y$ are self-adjoint
$g$-tuples, and let $\mathbb F_g^+$ denote the free monoid on $g$ letters. If
$w=i_1\cdots i_m$, write $Z^w=Z_{i_1}\cdots Z_{i_m}$. Fix $s>0$ and set
\begin{equation}\label{eq:holomorphic-monomial-map}
 \mathbf m_s(Z)=\big(s^{-|w|}Z^w\big)_{\varnothing\ne w\in\mathbb F_g^+},
 \qquad
 \Gamma_{\mathrm{hol},s}(Z)
 =\big(\operatorname{Re}s^{-|w|}Z^w,
       \operatorname{Im}s^{-|w|}Z^w\big)_{w\ne\varnothing}.
\end{equation}
The second tuple is the self-adjoint realification of the first. Equivalently,
it records every analytic monomial $Z^w$ and every coanalytic monomial
$(Z^w)^*$. The length-one coordinates recover $(X,Y)$, so the graph map is
injective.

The weights in \eqref{eq:holomorphic-monomial-map} are topological rather than
algebraic: any nonzero weights give the same $\Gamma$-pairs. Their role is to
keep the infinite tuple bounded. For example, with
\[
 \|Z\|_{\mathrm{row}}
 =\left\|\sum_{j=1}^gZ_jZ_j^*\right\|^{1/2},
\]
the monomial row belongs to the natural row-operator-space completion whenever
$\|Z\|_{\mathrm{row}}<s$, since
\[
 \sum_{|w|=m}s^{-2m}Z^w(Z^w)^*
 \preceq \left(\frac{\|Z\|_{\mathrm{row}}}{s}\right)^{2m}I.
\]
Thus, on a bounded set, one chooses $s$ beyond its row radius and below a
slightly larger radius of uniform analyticity. This radius gap places the
monomials in a common operator-space completion and ensures that analytic free
power series factor through $\mathbf m_s$ by completely bounded affine maps.

For a twice differentiable self-adjoint free function $f$, its nc complex
Hessian is
\begin{equation}\label{eq:complex-hessian}
 \Delta f(Z)[H]
 =\left.\frac{\partial^2}{\partial t\,\partial\overline t}
   f(Z+tH)\right|_{t=0}.
\end{equation}
We call $f$ \emph{free plurisubharmonic}, or \emph{free plush}, when
$\Delta f(Z)[H]\succeq0$ at every level for every direction $H$.

The following triangular test is the key observation. The full monomial graph turns
strictly upper and strictly lower triangular perturbations into admissible
compressions, and the two orientations separate the two halves of the complex
Hessian.

\begin{proposition}[Full-graph Jensen inequalities imply plushness]
\label{prop:intrinsic-hol-plush}
Let $D$ be an open complex free domain on which
$\Gamma_{\mathrm{hol},s}$ is defined, and let $f$ be a twice continuously
differentiable self-adjoint free function on $D$. If $f$ is intrinsically
$\Gamma_{\mathrm{hol},s}$-convex, then $f$ is free plurisubharmonic.
\end{proposition}

\begin{proof}
Fix $Z\in D_n$ and a direction $H\in M_n(\mathbb C)^g$. For real $t$ put
\begin{equation}\label{eq:triangular-hol-pairs}
 Z_j^{\downarrow}(t)=
 \begin{pmatrix}Z_j&0\\tH_j&Z_j\end{pmatrix},
 \qquad
 Z_j^{\uparrow}(t)=
 \begin{pmatrix}Z_j&tH_j\\0&Z_j\end{pmatrix},
 \qquad
 V\xi=\binom{\xi}{0}.
\end{equation}
The free-domain axioms give $Z\oplus Z\in D_{2n}$, so both triangular tuples
belong to $D_{2n}$ for sufficiently small $t$. For every analytic word $w$,
the matrices $(Z^{\downarrow}(t))^w$ and
$(Z^{\uparrow}(t))^w$ are respectively lower and upper triangular, with
$Z^w$ in both diagonal corners. Hence
\[
 V^*(Z^{\downarrow}(t))^wV=Z^w
   =V^*(Z^{\uparrow}(t))^wV.
\]
The corresponding identities for coanalytic words follow by taking adjoints.
Thus both $(Z^{\downarrow}(t),V)$ and
$(Z^{\uparrow}(t),V)$ are $\Gamma_{\mathrm{hol},s}$-pairs.

Intrinsic Jensen convexity now gives
\begin{equation}\label{eq:triangular-jensen}
 f(Z)\preceq V^*f(Z^{\downarrow}(t))V,
 \qquad
 f(Z)\preceq V^*f(Z^{\uparrow}(t))V.
\end{equation}
Write
\[
 K_{\downarrow,j}=\begin{pmatrix}0&0\\H_j&0\end{pmatrix},
 \qquad
 K_{\uparrow,j}=\begin{pmatrix}0&H_j\\0&0\end{pmatrix}.
\]
The second-order free Taylor expansion at $Z\oplus Z$ has the form
\begin{align}
 V^*f(Z\oplus Z+tK_\downarrow)V
   &=f(Z)+t^2Q_\downarrow(Z,H)+o(t^2),\label{eq:q-down}\\
 V^*f(Z\oplus Z+tK_\uparrow)V
   &=f(Z)+t^2Q_\uparrow(Z,H)+o(t^2).\label{eq:q-up}
\end{align}
There are no first-order diagonal terms. There are also no pure second-order
terms: their block factors contain either $E_{21}E_{21}$ or
$E_{12}E_{12}$, both zero. The first corner in
\eqref{eq:q-down} selects the mixed terms in which $H^*$ occurs before $H$,
while the first corner in \eqref{eq:q-up} selects those in which $H$ occurs
before $H^*$. Consequently, with the normalization in
\eqref{eq:complex-hessian},
\begin{equation}\label{eq:hessian-split}
 \Delta f(Z)[H]=Q_\downarrow(Z,H)+Q_\uparrow(Z,H).
\end{equation}
This identity is the matrix-unit decomposition of the mixed second
derivative; it may first be checked on nc monomials and then follows for a
$C^2$ free function from its second Fr\'echet differential.

Subtracting $f(Z)$ in \eqref{eq:triangular-jensen}, dividing by $t^2$, and
letting $t\to0$ shows that both $Q_\downarrow(Z,H)$ and
$Q_\uparrow(Z,H)$ are positive semidefinite. Equation
\eqref{eq:hessian-split} therefore gives $\Delta f(Z)[H]\succeq0$.
\end{proof}

\begin{proposition}[Convex holomorphic pullbacks are plush]
\label{prop:convex-pullback-plush}
Let $\Phi$ be a free holomorphic map and let $G$ be a twice differentiable
matrix-convex free function on a neighborhood of the self-adjoint
realification of $\Phi(D)$. Then
\[
 f(Z)=G\big(\operatorname{Re}\Phi(Z),
             \operatorname{Im}\Phi(Z)\big)
\]
is free plurisubharmonic on $D$. In particular, every lift-Cartan
$\Gamma_{\mathrm{hol},s}$-convex function is locally free
plurisubharmonic.
\end{proposition}

\begin{proof}
The mixed derivative of a holomorphic coordinate vanishes, as does that of
its coanalytic adjoint. The second-order chain rule therefore contains no
term involving a second mixed derivative of $\Phi$. Up to the inessential
normalization in \eqref{eq:complex-hessian}, it gives
\[
 \Delta f(Z)[H]
 =D^2G\big(\Gamma_\Phi(Z)\big)
   \big[D\Gamma_\Phi(Z)[H],D\Gamma_\Phi(Z)[H]\big],
\]
where $\Gamma_\Phi=(\operatorname{Re}\Phi,
\operatorname{Im}\Phi)$. The right side is positive because the real nc
Hessian of a matrix-convex free function is positive. The local assertion
follows by applying the calculation to a convex lift.
\end{proof}

The converse uses a structural theorem. For polynomials it is the hereditary
and antihereditary decomposition of
Greene, Helton, and Vinnikov \cite{GHV}, together with Greene's local-to-global
version \cite{Greene}. For rational functions it follows from
\cite{DHKMV}; the uniformly real analytic version is proved in
\cite{PascoePSH}.

\begin{theorem}[Full-graph convexity and plurisubharmonicity]
\label{thm:plush-cartan}
The following statements hold.
\begin{enumerate}[label=\textup{(\roman*)}]
\item For a self-adjoint nc polynomial $p(Z,Z^*)$, the following are
equivalent:
\begin{enumerate}[label=\textup{(\alph*)}]
\item $p$ is plush;
\item $p$ is intrinsically $\Gamma_{\mathrm{hol},s}$-convex;
\item $p$ is lift-Cartan $\Gamma_{\mathrm{hol},s}$-convex.
\end{enumerate}
Only finitely many of the monomial coordinates are needed for the convex lift
of a given $p$.
\item Let $f$ be a uniformly real analytic self-adjoint free function on a
uniform free domain containing a scalar center, translated to $0$. On every
smaller uniform row ball about $0$ whose closure lies in the domain of
analyticity, one may choose $s$ so that the following are equivalent:
$f$ is plush, $f$ is intrinsically $\Gamma_{\mathrm{hol},s}$-convex, and
$f$ is lift-Cartan $\Gamma_{\mathrm{hol},s}$-convex.
\end{enumerate}
\end{theorem}

\begin{proof}
Intrinsic $\Gamma_{\mathrm{hol},s}$-convexity implies plushness directly by
\cref{prop:intrinsic-hol-plush}. A convex lift implies plushness by
\cref{prop:convex-pullback-plush}. It remains to prove that plushness supplies
both kinds of $\Gamma$-convexity.

If $p$ is plush, the classification theorem of \cite{GHV} gives analytic nc
polynomials $a_j,b_k,h$ such that
\begin{equation}\label{eq:plush-sos}
 p(Z,Z^*)=
 \sum_j a_j(Z)^*a_j(Z)+
 \sum_k b_k(Z)b_k(Z)^*+h(Z)+h(Z)^*.
\end{equation}
Each analytic polynomial in \eqref{eq:plush-sos} is an affine linear function
of the finite monomial vector. If $A_j(W),B_k(W),H(W)$ denote the
corresponding affine functions of independent lifted coordinates, then
\[
 G(W)=\sum_jA_j(W)^*A_j(W)+
       \sum_kB_k(W)B_k(W)^*+H(W)+H(W)^*
\]
is matrix convex and $p=G\circ\Gamma_{\mathrm{hol},s}$. Indeed, its Hessian
is a sum of the positive squares of the affine differentials of $A_j$ and
$B_k$. This is a global convex lift, and hence proves the lift-Cartan
condition.

The same representation also proves intrinsic Jensen convexity. If $(Z,V)$ is a
$\Gamma_{\mathrm{hol},s}$-pair and $Y=V^*ZV$, then preservation of all
analytic moments gives $a_j(Y)=V^*a_j(Z)V$, and similarly for $b_k$ and $h$.
Since $VV^*\preceq I$,
\[
 a_j(Y)^*a_j(Y)\preceq V^*a_j(Z)^*a_j(Z)V,
 \qquad
 b_k(Y)b_k(Y)^*\preceq V^*b_k(Z)b_k(Z)^*V,
\]
while the $h+h^*$ term compresses with equality. Summing proves intrinsic
$\Gamma_{\mathrm{hol},s}$-convexity.

For the analytic statement, the local convex-composition theorem of
\cite{PascoePSH} represents a uniformly real analytic plush function as a
matrix-convex free function composed with an analytic free map. On a smaller
ball, every component of that analytic map has a uniformly convergent free
power series. Choosing the intermediate radius $s$ as above makes the map a
completely bounded affine function of $\mathbf m_s$. Composing the convex
outer function with this affine map gives a convex lift through
$\Gamma_{\mathrm{hol},s}$ on the smaller ball. Its matrix Jensen inequality
gives intrinsic $\Gamma_{\mathrm{hol},s}$-convexity there. The two converse
implications were established in the first paragraph.
\end{proof}

\begin{remark}[Regularity and the converse implication]
The implication
\[
 \text{intrinsic }\Gamma_{\mathrm{hol},s}\text{-convexity}
 \quad\Longrightarrow\quad
 \text{free plurisubharmonicity}
\]
uses only $C^2$ regularity and the openness of the free domain; no lifting or
representation theorem enters. The reverse implication in
\cref{thm:plush-cartan} uses the polynomial sum-of-squares theorem or the
uniformly analytic convex-composition theorem. We do not assert the reverse
for an arbitrary merely $C^2$ free function without such a structural
hypothesis.
\end{remark}

The analytic realization used in the preceding proof is
\begin{equation}\label{eq:pascoe-plush-realization}
 f(Z)=\operatorname{Re}g(Z)+
 \begin{bmatrix}v^+(Z)\\ v^-(Z)\end{bmatrix}^{*}
 \begin{bmatrix}I&-T(Z)\\-T(Z)^*&I\end{bmatrix}^{-1}
 \begin{bmatrix}v^+(Z)\\ v^-(Z)\end{bmatrix},
\end{equation}
where $g,v^+,T$ are analytic, $v^-$ is coanalytic, and $T$ is contractive
\cite{PascoePSH}. Formula \eqref{eq:pascoe-plush-realization} is a convex
butterfly once the analytic quantities are regarded as independent
coordinates. The full monomial graph provides a common coordinate system for
these quantities.

Combining \cref{thm:plush-cartan} with the $\Gamma$-Cartan extension theorem
gives the following continuation statement.

\begin{corollary}[Plurisubharmonic butterfly extension]
\label{cor:plush-extension}
Let $K$ be a bounded real free set of complex tuples, and
choose $s$ so that $\Gamma_{\mathrm{hol},s}(K)$ is bounded. Let $f$ be a
uniformly real analytic free plurisubharmonic function on a neighborhood of
$K$ for which the radius-gap hypothesis above holds. Then there are affine
free maps $\ell,\Lambda,P$ such that
\begin{equation}\label{eq:plush-butterfly}
 F(Z)=\ell(\Gamma_{\mathrm{hol},s}(Z))+
 \Lambda(\Gamma_{\mathrm{hol},s}(Z))^*
 P(\Gamma_{\mathrm{hol},s}(Z))^{-1}
 \Lambda(\Gamma_{\mathrm{hol},s}(Z))
\end{equation}
extends $f$ and is free plurisubharmonic on
\[
 \Omega_{\mathrm{hol}}
 =\{Z:P(\Gamma_{\mathrm{hol},s}(Z))\succ0\}.
\]
Moreover, $\Omega_{\mathrm{hol}}$ contains the full-holomorphic
$\Gamma$-hull
\[
 \Gamma_{\mathrm{hol},s}^{-1}\!\left(
   \ncco(\Gamma_{\mathrm{hol},s}(K))\right).
\]
The same conclusion holds for plush polynomials without the analytic
radius-gap hypothesis.
\end{corollary}

The scalar specialization recovers a classical hull. The weights do not affect
the moment equations, and therefore, for compact $K\subset\mathbb C^g$,
\begin{equation}\label{eq:scalar-holomorphic-hull}
 \begin{aligned}
 &z\in\Gamma_{\mathrm{hol},s}^{-1}\!\left(
       \overline{\operatorname{co}}\,\Gamma_{\mathrm{hol},s}(K)\right)\\
 &\qquad\Longleftrightarrow\quad
 \exists\mu\in\operatorname{Prob}(K):
 p(z)=\int_Kp(\zeta)\,d\mu(\zeta)
 \quad(p\in\mathbb C[z_1,\ldots,z_g]).
 \end{aligned}
\end{equation}
The right side is the representing-measure description of the polynomial hull
\cite{Gamelin}. Thus the scalar closed full-holomorphic $\Gamma$-hull is
$\widehat K$. Analytic and coanalytic monomial coordinates therefore turn an
ordinary convex hull in the lifted space into the polynomial hull in the
original variables. For the free plush class, the matrix Jensen inequality in
the lifted variables pulls back to the corresponding Jensen inequality for the
representing measures in
\eqref{eq:scalar-holomorphic-hull}.

\section{Standard \texorpdfstring{$\Gamma$-geometries}{Gamma-geometries} from partial convexity}
\label{sec:standard-gamma}

\subsection{Partial convexity}
Before turning to the saddle, it is useful to revisit two established
$\Gamma$-geometries. They show how much the choice of a single product
coordinate changes which compressions survive.

Split the variables as $(x,y)$ and take
\[
 \Gamma_{\mathrm{par}}(x,y)=(x,y,y^2).
\]
The equality
\[
 V^*Y^2V=(V^*YV)^2
\]
holds exactly when $\ran V$ reduces $Y$. It follows that
$\Gamma_{\mathrm{par}}$-convexity is convexity in $x$ with the parameter $y$
held fixed; see \cite[Proposition~4.1]{JKMMPGamma}. Theorem
\ref{thm:gamma-extension} gives, for every lift-Cartan partially convex
function, a realization
\[
 f(x,y)=\ell(x,y,y^2)+\Lambda(x,y,y^2)^*
 P(x,y,y^2)^{-1}\Lambda(x,y,y^2)
\]
on a neighborhood of the partially convex hull. The lift-Cartan qualification
is essential by \cref{prop:partial-nonlift}. General partially convex rational
functions instead have the sharper root-butterfly realization of
\cite{JKMMPPartial}, whose
parameter-dependent affine tail contains the mixed-Hessian freedom responsible
for the obstruction.

\subsection{\texorpdfstring{$xy$-convexity}{xy-convexity} and bilinear matrix inequalities}
The standard self-adjoint encoding of the non-self-adjoint product $xy$ is
\begin{equation}\label{eq:gamma-xy}
 \Gamma_{xy}(x,y)=\big(x,y,xy+yx,i(xy-yx)\big).
\end{equation}
For self-adjoint $X,Y$, condition \eqref{eq:gamma-pair} is equivalent to
\[
 V^*XYV=(V^*XV)(V^*YV).
\]
This is the $xy$-pair condition of
\cite{JKMMPGamma,JKMMPPartial}. Affine pencils in the coordinates
\eqref{eq:gamma-xy} are precisely self-adjoint bilinear matrix pencils after
separating real and imaginary parts. Thus \cref{thm:gamma-extension} produces
analytic butterfly certificates on neighborhoods of $xy$-convex hulls, while
the polynomial and rational results of \cite{JKMMPPartial} give strong global
algebraic restrictions in their respective categories.

The next section omits the skew-adjoint product coordinate. Its absence permits
nonreducing compressions and produces a different scalar geometry.

\section{Jordan-product convexity and hyperbolic hulls}
\label{sec:jordan}

Let
\begin{equation}\label{eq:gamma-jordan}
 \Gamma_J(x,y)=(x,y,xy+yx).
\end{equation}
We call the associated notion \emph{Jordan-product convexity}. It differs from
standard $xy$-convexity by one missing coordinate: the skew-adjoint part of
$xy$ is no longer preserved.

\subsection{The compression condition}

\begin{proposition}[Jordan covariance]\label{prop:jordan-covariance}
Let $X,Y\in B(\Kk)_{\mathrm{sa}}$, let $V:\Hh\to\Kk$ be an isometry, and put
\[
 P=VV^*,\qquad A=(I-P)XV,\qquad B=(I-P)YV.
\]
Then $((X,Y),V)$ is a $\Gamma_J$-pair if and only if
\begin{equation}\label{eq:jordan-leakage}
 A^*B+B^*A=0.
\end{equation}
For the full $\Gamma_{xy}$ in \eqref{eq:gamma-xy}, the corresponding condition
is $A^*B=0$.
\end{proposition}

\begin{proof}
Write $P+(I-P)=I$ between the two factors. Then
\begin{align*}
 V^*(XY+YX)V
 &=(V^*XV)(V^*YV)+(V^*YV)(V^*XV)\\
 &\quad+V^*X(I-P)YV+V^*Y(I-P)XV\\
 &=(V^*XV)(V^*YV)+(V^*YV)(V^*XV)+A^*B+B^*A.
\end{align*}
This proves the first assertion. Preserving also $i(XY-YX)$ gives
$i(A^*B-B^*A)=0$; together with \eqref{eq:jordan-leakage}, this is equivalent
to $A^*B=0$.
\end{proof}

Condition \eqref{eq:jordan-leakage} does not force either off-diagonal operator
to vanish. Their contributions may cancel in the Jordan product.

\begin{example}[A nonreducing $\Gamma_J$-pair]\label{ex:3by3}
On $\mathbb C^3$, let $V:\mathbb C\to\mathbb C^3$ have range
$\mathbb Ce_1$ and set
\[
 X=E_{12}+E_{21},\qquad Y=E_{13}+E_{31}.
\]
Then $A=e_2$ and $B=e_3$, so $A^*B+B^*A=0$. Hence
$((X,Y),V)$ is a $\Gamma_J$-pair, although $\mathbb Ce_1$ reduces neither
$X$ nor $Y$. This gives a nonreducing admissible compression in dimension
three.
\end{example}

For an nc convex combination, the same computation yields the moment form of
the condition. If $\sum_iV_i^*V_i=I$ and
\[
 \bar X=\sum_iV_i^*X_iV_i,\qquad
 \bar Y=\sum_iV_i^*Y_iV_i,
\]
then the combination is admissible for $\Gamma_J$ precisely when
\begin{equation}\label{eq:matrix-jordan-moment}
 \sum_iV_i^*(X_iY_i+Y_iX_i)V_i
 =\bar X\bar Y+\bar Y\bar X.
\end{equation}
Thus the $\Gamma_J$-hull is characterized by vanishing Jordan covariance.

\subsection{The scalar saddle}
At level one,
\[
 \Gamma_J(x,y)=(x,y,2xy).
\]
Its image is the hyperbolic paraboloid
\[
 \Sigma=\{(x,y,z)\in\mathbb R^3:z=2xy\}.
\]
The graph-hull identity specializes as follows.

\begin{proposition}[Zero-covariance barycenters]\label{prop:scalar-j-hull}
For $S\subseteq\mathbb R^2$, a point $(\bar x,\bar y)$ belongs to the scalar
$\Gamma_J$-convex hull generated using scalar points of $S$ if and only if
there are points $(x_i,y_i)\in S$ and probabilities $t_i$ such that
\begin{equation}\label{eq:zero-covariance}
 \bar x=\sum_it_ix_i,\qquad
 \bar y=\sum_it_iy_i,\qquad
 \sum_it_ix_iy_i=\bar x\bar y.
\end{equation}
Equivalently, the random variables $x$ and $y$ have covariance zero.
Moreover,
\begin{equation}\label{eq:scalar-saddle-hull}
 \Gamma_J\text{-}\operatorname{co}(S)
 =\{(x,y):(x,y,2xy)\in\operatorname{co}(\Gamma_J(S))\}.
\end{equation}
\end{proposition}

\begin{proof}
Equation \eqref{eq:zero-covariance} is exactly the assertion that the convex
combination of the lifted points $(x_i,y_i,2x_iy_i)$ lies again on $\Sigma$.
This is \eqref{eq:scalar-saddle-hull}.
\end{proof}

For a two-point combination with weight $0<t<1$, the covariance condition is
\begin{equation}\label{eq:two-point-condition}
 t(1-t)(x_1-x_2)(y_1-y_2)=0.
\end{equation}
Thus nontrivial two-point $\Gamma_J$-segments are horizontal or vertical. The
next example shows that four points can generate an entire square.

\begin{example}[Four corners fill the square]\label{ex:four-corners}
Let $S=\{(\sigma,\tau):\sigma,\tau\in\{-1,1\}\}$. Given
$(x,y)\in[-1,1]^2$, assign the product weights
\[
 t_{\sigma,\tau}=\frac14(1+\sigma x)(1+\tau y).
\]
Then $\mathbb E\sigma=x$, $\mathbb E\tau=y$, and
$\mathbb E(\sigma\tau)=xy$. Hence the scalar $\Gamma_J$-hull of the four
corners is the entire square. By contrast, the two opposite points
$(1,1)$ and $(-1,-1)$ generate no interior segment, by
\eqref{eq:two-point-condition}.
\end{example}

\begin{proposition}[Sharp scalar Carath\'eodory number]
\label{prop:jordan-caratheodory}
Every point of a scalar $\Gamma_J$-hull has a zero-covariance representation
supported on at most four generating points. Four is best possible.
\end{proposition}

\begin{proof}
The lifted points $\Gamma_J(S)$ lie in $\mathbb R^3$, so ordinary
Carath\'eodory gives a representing convex combination with at most four
atoms. For sharpness, use the four corners from \cref{ex:four-corners}. Their
lifts
\[
 (1,1,2),\quad(1,-1,-2),\quad(-1,1,-2),\quad(-1,-1,2)
\]
are affinely independent, and the origin is their barycenter with all four
weights equal to $1/4$. It is therefore in the interior of the resulting
tetrahedron and in the convex hull of no three of its vertices. Hence the
zero-covariance representation of $(0,0)$ from this generating set requires
all four atoms.
\end{proof}

\subsection{Hyperbolic separation}
An affine half-space in the lifted variables has the form
\[
 a+bx+cy+dz\geq0.
\]
Its pullback to $\Sigma$ is
\begin{equation}\label{eq:hyperbolic-halfspace}
 a+bx+cy+2dxy\geq0.
\end{equation}
When $d\neq0$, the boundary can be written
\[
 \left(x+\frac{c}{2d}\right)
 \left(y+\frac{b}{2d}\right)
 =\frac{bc-2ad}{4d^2},
\]
a rectangular hyperbola, possibly degenerate. When $d=0$ it is a line.

\begin{proposition}[Hyperbolic half-space description]
\label{prop:hyperbolic-separation}
For $S\subseteq\mathbb R^2$, the closed lifted scalar hull
\[
 \{(x,y):\Gamma_J(x,y)\in
          \overline{\operatorname{co}}(\Gamma_J(S))\}
\]
is the intersection of all regions \eqref{eq:hyperbolic-halfspace} containing
$S$. If $\operatorname{co}(\Gamma_J(S))$ is closed---in particular, if $S$ is
compact---this is the scalar $\Gamma_J$-convex hull itself.
\end{proposition}

\begin{proof}
By \eqref{eq:scalar-saddle-hull}, membership is membership of the lifted point
in the closed ordinary convex hull of $\Gamma_J(S)$. The ordinary Hahn--Banach
separation theorem describes that convex hull as the intersection of its
affine half-spaces. Pulling them back gives
\eqref{eq:hyperbolic-halfspace}.
\end{proof}

At matrix levels, the same construction uses operator coefficients. A
$\Gamma_J$-pencil is
\begin{equation}\label{eq:jordan-pencil}
 L(X,Y)=A_0\otimes I+A_x\otimes X+A_y\otimes Y
 +A_J\otimes(XY+YX).
\end{equation}
Its positivity set is $\Gamma_J$-convex. At level one, positivity of a
matrix-valued pencil is the intersection, over coefficient-space vectors
$v$, of scalar hyperbolic half-spaces with coefficients
$v^*A_\bullet v$. The determinant boundary of a matrix-valued pencil can of
course have degree larger than two. The hyperbolas are the supporting scalar
inequalities; the determinant of a matrix-valued pencil need not be quadratic.

\subsection{Jordan-product butterfly realizations}
The $\Gamma$-Cartan theorem gives the following continuation result.

\begin{corollary}[Jordan-product butterfly]\label{cor:jordan-butterfly}
Let $K$ be a bounded real free set of pairs and let $f$ be lift-Cartan
$\Gamma_J$-convex on $K$. Then $f$ extends to a
$\Gamma_J$-convex neighborhood of $\Gamma_J\text{-}\ncco(K)$ and admits
the following realization, where $Z_J=(X,Y,XY+YX)$:
\[
 F(X,Y)=\ell(Z_J)+\Lambda(Z_J)^*P(Z_J)^{-1}\Lambda(Z_J).
\]
The denominator positivity domain and the epigraph are described by
$\Gamma_J$-pencils.
\end{corollary}

Two elementary examples are
\[
 F(X,Y)=(XY+YX)^2
\]
and
\[
 F(X,Y)=\big(I-(XY+YX)\big)^{-1},
 \qquad I-(XY+YX)\succ0.
\]
They are pullbacks of the matrix-convex one-variable functions $s^2$ and
$(I-s)^{-1}$, respectively. Their dependence on the original variables is
represented through a single Jordan-product coordinate.

\section{Envelope interpretation and further problems}
\label{sec:outlook}

Theorems \ref{thm:cartan-extension} and \ref{thm:gamma-extension} are
extension statements formulated through butterfly realizations. A local
convex germ at the universal point produces an affine pencil $P$, and
positivity of $P$ propagates to every compression of every amplification of
that point. The domain $\{P\succ0\}$ is therefore a natural
continuation domain attached to the chosen realization. It need not be maximal
or unique, but it contains the relevant hull with a uniform positive margin.

This point of view is complementary to several existing envelope theories.
Effros--Winkler separation describes closed matrix-convex sets by pencils
\cite{EW}; Webster--Winkler and Davidson--Kennedy identify compact nc convex
sets with operator systems \cite{WW,DK}; and the $\Gamma$-theory describes
partial hulls by lifting and pulling back \cite{JKMMPGamma,KMS}. Here the
epigraph pencil records function values as well as set geometry, while the
royal-road theorem supplies analyticity and a resolvent formula.

We first record two elementary limitations of the extension theorem.

\begin{proposition}[Compactness does not replace all levels]
\label{prop:finite-level-obstruction}
Local convex liftability and compactness at finitely many levels do not imply
extension to the convex hull.
\end{proposition}

\begin{proof}
Already at level one, take
\[
 K=\{-1,0,1\},\qquad f(-1)=f(1)=0,\qquad f(0)=1.
\]
The datum is compact and is locally the restriction of a constant, hence
convex, function near each of its three isolated points. If it extended to a
convex function on a neighborhood of $[-1,1]$, however, midpoint convexity
would give
\[
 1=f(0)\leq\frac{f(-1)+f(1)}2=0,
\]
a contradiction. Thus compactness cannot replace the compatibility supplied
by a universal direct sum.
\end{proof}

\begin{proposition}[The Cartan data do not select a continuation]
\label{prop:no-canonical-continuation}
There need not be a nontrivial neighborhood on which all matrix-convex lifts
of fixed Cartan data agree.
\end{proposition}

\begin{proof}
In one variable let $K$ consist of the zero matrix at every level and put
$f(0)=0$. Both
\[
 G_0(X)=0,\qquad G_2(X)=X^2
\]
are globally matrix-convex free functions extending $f$, but they disagree
at every nonzero point. Hence the values $(K,f)$ alone cannot determine a
maximal common continuation domain.
\end{proof}

The following questions remain open.

\begin{question}[Truncated Jensen completion]
Fix $N$ and let $K$ be compact truncated free data through level $N$. Suppose
$f$ satisfies every matrix Jensen inequality witnessed by direct sums and
compressions that remain at levels at most $N$. Is this finite collection of
compatibility conditions sufficient for a matrix-convex extension through
level $N$ to a neighborhood of the truncated hull? If not, what additional
boundary-level compatibility is necessary?
\end{question}

\begin{question}[Complete normal curvature]
Let $\Gamma$ be a real analytic free embedding. Is local
intrinsic-to-lift equivalence characterized by its complete second fundamental
form? More precisely, is the relevant condition that, for every Hessian
$q_A(H)$ allowed by the intrinsic $\Gamma$-pair tests, there is a
self-adjoint completely bounded normal multiplier $L_A$ such that
\[
 q_A(H)+L_A\!\left(D^2\Gamma(A)[H,H]\right)\succeq0
\]
completely, with an order margin sufficient to control mixed normal
directions? The quadratic convexification lemma verifies this mechanism for
$\Gamma_\square$, while the full holomorphic graph is governed instead by the
two triangular Hessian tests. Proposition~\ref{prop:partial-nonlift} shows
that positivity in the normal curvature alone cannot suffice: the criterion
must also force the mixed Hessian to vanish on the kernel of every uncorrected
tangential block.
\end{question}

\begin{question}[Minimal-realization envelope]
Fix a Cartan germ at the universal point, rather than only its values on
$K$, and choose a minimal butterfly realization. Is the component of
$\{P\succ0\}$ containing the germ invariant, up to realization congruence,
and is it the maximal single-valued matrix-convex continuation domain? This is
the formulation for which the canonical rational domains in
\cite{HMV,JKMMPPartial} provide a genuine model.
\end{question}

\begin{question}[Matrix Jordan boundary]
Determine the matrix extreme and absolute extreme points of the
$\Gamma_J$-hull, together with optimal level-dependent matrix
Carath\'eodory bounds. The scalar Carath\'eodory number is exactly four by
\cref{prop:jordan-caratheodory}; at matrix levels the operator-valued
covariance constraint \eqref{eq:matrix-jordan-moment} and the nonreducing pair
in \cref{ex:3by3} show that the boundary problem is genuinely different.
\end{question}

The cardinally bounded universal sum places the local compatibility
conditions at a single point. The butterfly theorem then controls the analytic
continuation, while the chosen $\Gamma$-graph determines the corresponding
hull geometry.

\section*{Statement on AI-assisted preparation}

The methodology of this paper is described in
\cref{sec:ai-methodology}. We supplied the mathematical koan, the central
definitions and examples, and the direction of the argument. OpenAI's
ChatGPT was used to expand that koan into candidate proofs and expositions, to
organize the \LaTeX{} draft, and to assist with calculations, references, and
copy editing. We repeatedly corrected and restricted the generated claims,
reviewed the resulting manuscript, and accept responsibility for all
mathematical assertions, proofs, citations, and remaining errors.

\end{document}